\documentclass[11pt]{amsart}
\usepackage{amsthm}
\usepackage{amscd}  
\usepackage{graphicx}
\usepackage{amssymb}
\usepackage{epstopdf}
  
\theoremstyle{plain}
\newtheorem{Thm}{Theorem}[section]

\newtheorem{Prop}[Thm]{Proposition}

\newtheorem{Lem}[Thm]{Lemma}

\theoremstyle{definition}
\newtheorem{Defn}[Thm]{Definition}

\newtheorem{Rem}[Thm]{Remark}
 
\numberwithin{equation}{section}

\title{On formal non-commutative deformations of smooth varieties}
\author{Yujiro Kawamata}
\usepackage{fancyhdr}
\begin{document}
\maketitle

\tableofcontents

\begin{abstract}
We will develop a formal non-commutative (NC) deformation theory of smooth algebraic varieties $X$
defined over a field $k$, and describe a semi-universal deformation where the tangent space $T^1$ and the
obstruction space $T^2$ are given by the Hochschild cohomology groups.

14D15, 16E40.
\end{abstract}

\section{Introduction}

We will develop a formal non-commutative (NC) deformation theory of smooth algebraic varieties $X$
defined over a field $k$.
This is an extension of the theory in \cite{Toda} where the first order deformations are considered. 
We consider two versions of NC deformations, non-twisted and twisted NC deformations.
The former is a subfunctor of the letter and treats only deformations of the variety $X$ itself, while the letter
treats deformations of pairs consisting of varieties and the system of twistings.
The pair determines a category of twisted sheaves, but we do not consider the deformations of categories themselves.

We will describe infinitesimal extensions of NC deformations by using the {\em tangent space} $T^1$ and 
the {\em obstruction space} $T^2$ of the deformation functor.
They are given by the Hochschild-Kostant-Rosenberg decompositions of the Hochschild cohomologies in the
case of twisted NC deformations:
\[
\begin{split}
&T^1 \cong H^0(X, \bigwedge^2 T_X) \oplus H^1(X, T_X) \oplus H^2(X, \mathcal O_X), \\
&T^2 \cong H^0(X, \bigwedge^3 T_X) \oplus H^1(X, \bigwedge^2 T_X) \oplus H^2(X, T_X) \oplus H^3(X, \mathcal O_X),
\end{split}
\] 
where $T_X$ is the tangent sheaf of $X$ (Theorem \ref{T1T2'}).
The case of non-twisted NC deformations corresponds to their subspaces consisting of those factors 
except $H^2(X, \mathcal O_X)$ and $H^3(X, \mathcal O_X)$ (Theorem \ref{T1T2}).

We will prove that our NC deformation functors satisfy the conditions of Schlessinger (\cite{Schlessinger}) 
so that it has a semi-universal formal deformation (Theorems \ref{semi-universal} and \ref{semi-universal'}).

The following are the main results in the case of twisted NC deformations:

\begin{Thm}[= Theorem \ref{T1T2'}]
Let $X$ be a smooth variety, let $0 \to J \to R' \to R \to 0$ be a small extension of Artin local algebras 
(i.e., $J$ is killed by the maximal ideal), and 
let $\mathcal A$ be a twisted NC deformation of $X$ over $R$.
Then the following hold.

(1) There is a class $\xi^{(3,0)} \in J \otimes H^0(X, \bigwedge^3 T_X)$.
If $\xi^{(3,0)} = 0$, then there is a class $\xi^{(2,1)} \in J \otimes H^1(X, \bigwedge^2 T_X)$.
If $\xi^{(2,1)} = 0$, then there is a class $\xi^{(0,3)} \in J \otimes H^3(X, \mathcal O_X)$.
If $\xi^{(0,3)} = 0$, then there is a class $\xi^{(1,2)} \in J \otimes H^2(X, T_X)$.
And $\xi^{(3,0)} = \xi^{(2,1)} = \xi^{(1,2)} = \xi^{(0,3)} = 0$ hold if and only if there exists a twisted NC deformation $\mathcal A'$
over $R'$ which induces $\mathcal A$ over $R$.

(2) Assume that $\xi^{(3,0)} = \xi^{(2,1)} = \xi^{(1,2)} = \xi^{(0,3)} = 0$.
Then the set of isomorphism classes of all extensions $\mathcal A'$ 
corresponds bijectively to a vector space $J \otimes (H^0(X, \bigwedge^2 T_X) \oplus H^1(X, T_X) \oplus H^2(X, \mathcal O_X))$.

(3) The obstruction classes $\xi^{(3,0)}, \xi^{(2,1)}, \xi^{(1,2)}, \xi^{(0,3)}$ are functorially determined with respect to 
commutative diagrams of small extensions
\begin{equation}
\begin{CD}
0 @>>> J @>>> R' @>>> R @>>> 0 \\
@. @V{\beta_J}VV @V{\beta'}VV @V{\beta}VV \\
0 @>>> J_1 @>>> R'_1 @>>> R_1 @>>> 0.
\end{CD}
\end{equation}
\end{Thm}

\begin{Thm}[= Theorem \ref{semi-universal'}]
Assume that $n = \dim T^1$ and $m = \dim T^2$ are finite.
Then there exists a semi-universal twisted NC deformation of $X$ over 
\[
R := k[[t_1,\dots,t_n]]/(f_1,\dots,f_m)
\]
for some formal power series $f_i$ of orders at least $2$ and possibly vanish.
\end{Thm}

A general theory of infinitesimal deformations of abelian categories such as $\text{coh}(X)$ 
are developed in \cite{Lowen-VdBergh}.
The tangent space and the obstruction space are identified as Hochschild cohomologies 
$HH^2(X)$ and $HH^3(X)$.
We note that there is Hochschild-Kostant-Rosenberg isomorphism
\[
HH^n(X) \cong \bigoplus_{p+q=n} H^q(X, \bigwedge^p T_X)
\]
in the case of a smooth projective variety.
On the other hand, \cite{Toda} considered NC deformations of smooth projective varieties over the ring $k[t]/(t^2)$ explicitly.
Our result is its generalization to higher order deformations.
We use explicit gluing of local charts in a similar way to Kodaira.  

\vskip 1pc

The outline of the proof is as follows.
We start with the non-twisted version in \S 2.
We cover the variety $X$ by affine open subsets $U_i = \text{Spec}(A_i)$, 
deform the coordinate rings $A_i$ by changing the multiplications, then glue them.
The gluing is done only by homomorphisms between NC deformed coordinate rings $\mathcal A_i \to \mathcal A_j$
which are related by inclusions $U_j \subset U_i$, because localization is 
not possible in general due to the non-commutativity of the associative rings $\mathcal A_i$.
We cannot use a sheaf for gluing.

The associativity of the multiplication, the compatibility of the multiplication with the gluing, and the transitivity 
of the gluing are translated to the vanishing of certain \v Cech cocycles of Hochschild cocycles.
Namely they correspond to
a \v Cech $0$-cocycle of Hochschild $3$-cocycles, a \v Cech $1$-cocycle of Hochschild $2$-cocycles, and
a \v Cech $2$-cocycle of Hochschild $1$-cocycles, respectively.
They are first defined only for configurations of open subsets which are related by inclusions.
Moreover we need to consider Hochschild cocycles with coefficients.
For example, we consider a Hochschild cocycle on $A_i$ with coefficients in $J \otimes A_j$ for $U_j \subset U_i$, and 
a \v Cech $2$-cocycle is defined on the overlap $U_i \cap U_j \cap U_k$ only in the case where
$U_k \subset U_j \subset U_i$.
Then we extend the \v Cech cocycles for general triple $U_i,U_j,U_k$, which are unrelated by inclusions, using some trick
(Lemma \ref{S_n}).

In \S 3, we extend the results to the twisted version.
We consider modules twisted by invertible functions with respect to the transitivity of open subsets $U_k \subset U_j \subset U_i$.
Since the twisting functions do not necessarily belong to the centers of the function algebras, we need to require that the 
function algebras are also twisted by the adjoint action of the twisting functions.
We introduce the equivalence condition of the twistings which is compatible with the formation of the category of twisted modules. 
The compatibility condition of the twisting with gluing yields a \v Cech $3$-cocycle of 
Hochschild $0$-cocycles, i.e., functions.
We will explain the necessary changes for the generalization from the non-twisted version to the twisted version.

The author would like to thank Alexey Bondal, Anya Nordskova, Shinnosuke Okawa and Michel Van den Bergh for 
suggestions and encouragements.
He would also like to thank NCTS of National Taiwan University where 
the work was partly done while the author visited there.
This work is partly supported by JSPS Kakenhi 21H00970.

\section{NC deformations of smooth varieties}

We consider NC deformations without twisting in this section.
This simplifies the notation and arguments.
The twisted case will be treated in the next section. 

We define a concept of an NC scheme considered in this paper:

\begin{Defn}
An {\em NC scheme} $\mathcal A$ consists of the following data:

\begin{enumerate}
\item A poset $I$ on which the maximum $\max\{i,j\}$ exists for any two elements $i,j \in I$;

\item A set of associative algebras $\mathcal A_i$ parametrized by $i \in I$.

\item Algebra homomorphisms $\phi_{ji}: \mathcal A_i \to \mathcal A_j$ for 
any pair $i < j$
(we set $\phi_{ii} = \text{Id}_{\mathcal A_i}$), which are transitive, i.e., 
$\phi_{kj}\phi_{ji}(x) = \phi_{ki}(x)$ for $i < j < k$ and $x \in \mathcal A_i$.
\end{enumerate}
\end{Defn}

We usually assume the following condition (flatness and the birationality), which will be useful 
when we consider coherent sheaves on NC schemes (cf. \cite{DVLL}):

\begin{itemize}
\item $\phi_{ji}$ is two-sided flat, and the natural map $\mathcal A_j \to \mathcal A_j \otimes_{\mathcal A_i} \mathcal A_j$ 
is bijective.
\end{itemize}

An algebraic variety $X$ defined over a field $k$ is an example of an NC scheme.
$X$ has an open covering by affine open subsets $U_i = \text{Spec }A_i$ for $i \in I$.
We write $i < j$ if $U_i \supset U_j$.
We assume that, if $i,j \in I$, then $U_i \cap U_j = U_k$ for some $k \in I$.
We have restriction homomorphisms $\phi^A_{ji}: A_i \to A_j$ for $i < j$, which is flat.
Since $A_i$ and $A_j$ are birationally equivalent, the natural map 
$A_j \to  A_j \otimes_{A_i} A_j$ is bijective.

\vskip 1pc

Let $(\text{Art})$ be the category of Artin local $k$-algebras $(R,M)$ such that $R/M \cong k$.
We can write $R = k[[ t_1,\dots,t_n]]/I$ for some integer $n$, where $I$ is an ideal of 
a formal power series ring.
$R$ is finite dimensional as a $k$-module, 
and the maximal ideal is given by $M = (t_1,\dots,t_n)$. 

\begin{Defn}
Let $X$ be an algebraic variety defined over a field $k$ with a fixed covering 
by affine open subsets $U_i = \text{Spec }A_i$ for $i \in I$ as above.
An {\em NC deformation} of $X$ over $(R, M)$ is defined to be
a pair $(\mathcal A, \alpha)$ which satisfies the following conditions:
 
\begin{enumerate}

\item $\mathcal A$ is an NC scheme consisting of associative $R$-algebras $\mathcal A_i$ for $i \in I$ 
which are flat over $R$, and a collection of $R$-algebra homomorphisms
$\phi^{\mathcal A}_{ji}: \mathcal A_i \to \mathcal A_j$ for $i < j$ which are transitive, i.e., 
$\phi^{\mathcal A}_{kj} \circ \phi^{\mathcal A}_{ji} = \phi^{\mathcal A}_{ki}$ for $i < j < k$.
We set $\phi^{\mathcal A}_{ii} = \text{Id}_{\mathcal A_i}$ for $i = j$.

\item $\alpha$ is a collection of $k$-algebra isomorphisms $\alpha_i: k \otimes_R \mathcal A_i \to A_i$, 
which is compatible with the $\phi^{\mathcal A}_{ji}$, i.e., 
$\alpha_j \circ (k \otimes \phi^{\mathcal A}_{ji}) = \phi^A_{ji} \circ \alpha_i$.
\end{enumerate}
\end{Defn}

We do not need to assume the flatness and the birationality condition on gluing homomorphisms $\phi^{\mathcal A}_{ji}$
in the case of NC deformations of algebraic varieties, because they are automatic:

\begin{Prop}\label{flat birational}
Let $X$ be an algebraic variety over a field $k$ with a fixed open affine covering $U_i = \text{Spec}(A_i)$ as before.
Let $\mathcal A$ be an NC deformation of $X$ over $R \in (\text{Art})$.
Then the following hold.

(1) $\mathcal A_j$ is right and left flat over $\mathcal A_i$ for $i < j$. 

(2) The natural $\mathcal A_j$-module homomorphism $\mathcal A_j \to \mathcal A_j \otimes_{\mathcal A_i} \mathcal A_j$ 
is bijective (a birationality condition). 
\end{Prop}

\begin{proof}
(1) We will prove the right flatness.
The left flatness is similarly proved.
We need to prove that $\text{Tor}_m^{\mathcal A_i}(\mathcal A_j, M) \cong 0$ for any left 
$\mathcal A_i$-module $M$ and $m > 0$.
Since $\dim R < \infty$, we may assume that $M$ is an $A_i$-module.

Let $\mathcal P_{\bullet} \to \mathcal A_j$ with $\mathcal P_m = \mathcal A_i^{(I_m)}$ for $m \ge 0$ be a resolution
by free $\mathcal A_i$-modules.
Since $\mathcal A_j$ and $\mathcal A_i$ are flat over $R$, the exactness is preserved under the base change $k \otimes_R$; 
$P_{\bullet} \to A_j$ with $P_m = A_i^{(I_m)}$ is also a resolution by free $A_i$-modules. 
Since $A_j$ is flat over $A_i$, we have $\text{Tor}_m^{A_i}(A_j, M) \cong 0$ for $m > 0$.
Hence $P_{\bullet} \otimes_{A_i} M \to A_j \otimes_{A_i} M$ is a resolution.
We have 
\[
\mathcal A_j \otimes_{\mathcal A_i} M \cong A_j \otimes_{A_i} M
\]
Hence $\mathcal P_{\bullet} \otimes_{\mathcal A_i} M \to \mathcal A_j \otimes_{\mathcal A_i} M$ is also a resolution.
Terefore $\text{Tor}_m^{\mathcal A_i}(\mathcal A_j, M) \cong 0$ for $m > 0$.

\vskip 1pc

(2) We proceed by induction on the dimension of the base ring $\dim R$.
Assuming that a deformation $\mathcal A$ over $R$ satisfies the birationality condition, 
we will prove that an extension $\mathcal A'$ of $\mathcal A$ over a small extension $R'$ of $R$ also satisfies it.

The first step is that, since the natural homomorphism $A_i \to A_j$ is birational, 
$A_j = A_i \otimes_{A_i} A_j \to A_j \otimes_{A_i} A_j$ is bijective.

Our homomorphism is derived from $\mathcal A_i \to \mathcal A_j$
by $\mathcal A_j = \mathcal A_i \otimes_{\mathcal A_i} \mathcal A_j$.
We have also $\mathcal A'_j \otimes_{\mathcal A'_i} \mathcal A_j \cong \mathcal A_j \otimes_{\mathcal A_i} \mathcal A_j$
and $\mathcal A'_j \otimes_{\mathcal A'_i} A_j \cong A_j \otimes_{A_i} A_j$.
Therefore we have a commutative diagram of exact sequences
\[
\begin{CD}
0 @>>> J \otimes A_j @>>> \mathcal A'_j @>>> \mathcal A_j @>>> 0 \\
@. @VVV @VVV @VVV \\
0 @>>> J \otimes A_j \otimes_{A_i} A_j @>>> \mathcal A'_j \otimes_{\mathcal A'_i} \mathcal A'_j 
@>>> \mathcal A_j \otimes_{\mathcal A_i} \mathcal A_j @>>> 0.
\end{CD}
\]
Since the vertical arrows on both sides are bijective, so is the middle. 
\end{proof}

Let $(R', M') \in (\text{Art})$ be another Artin local $k$-algebra and let $R \to R'$ be a local $k$-algebra homomorphism.
Then an NC deformation 
$\mathcal A = (\mathcal A, \alpha)$ of $X$ over $R$ induces an NC deformation 
$\mathcal A' = (\mathcal A', \alpha')$ over $R'$ by 
$\mathcal A'_i = R' \otimes_R \mathcal A_i$ for all $i$.
Thus the NC deformations of $X$ determines a functor $\text{Def}_X: (\text{Art}) \to (\text{Set})$ 
given by $\text{Def}_X(R) = \{(\mathcal A, \alpha)\}/\sim$, where $\sim$ denotes an isomorphism.

\vskip 1pc

In order to describe the deformation functor using $T^1$ and $T^2$, we recall Hochschild cohomologies (cf. \cite{W}).
Let $A$ be an associative $k$-algebra and let $M$ be an $A$-bimodule.
The Hochschild cohomology $HH^*(A,M)$ of $A$ with coefficients in $M$ is the cohomology of the cochain complex
\[
\begin{split}
&0 \to M \to \text{Hom}_k(A, M) \to \dots \\
&\to \text{Hom}_k(A^{\otimes p}, M) \to \text{Hom}_k(A^{\otimes (p+1)}, M) \to \dots
\end{split}
\]
where $A^{\otimes p} = A \otimes_k \dots \otimes_k A$ ($p$-times) and the coboundary is given by
\[
\begin{split}
&df(x_1, \dots, x_p) 
= x_1f(x_2,\dots,x_p) \\&+ \sum_{i = 1}^{p-1} (-1)^i f(x_1, \dots, x_ix_{i+1}, \dots, x_p)
+ (-1)^p f(x_1, \dots, x_{p-1})x_p.
\end{split}
\]

The Hochschild cohomology is particularly simple when $A$ is a commutative and $X = \text{Spec }A$ is smooth.
Indeed if $M$ is a flat $A$-module, then we have 
$HH^p(A,M) \cong \bigwedge^p T_X \otimes_A M$, because there is a quasi-isomorphism
\[
\begin{CD}
@>>> A \otimes A^{\otimes p} @>d>> A \otimes A^{\otimes (p-1)} \dots @>d>> A \otimes A @>0>> A @>>> 0 \\
@. @V{I^p}VV @V{I^{p-1}}VV @V{I^1}VV @V{I^0}VV \\
@>>> \Omega_A^p @>0>> \Omega_A^{p-1} \dots @>0>> \Omega^1_A @>0>> A @>>> 0
\end{CD}
\]
given by 
\[
I^p(x_0 \otimes x_1 \otimes \dots \otimes x_p) = 
x_0 dx_1 \wedge \dots \wedge dx_p.
\]
Here we used $\text{Hom}_k(A^{\otimes p}, M) \cong \text{Hom}_A(A \otimes_k A^{\otimes p}, M)$.

We note that the last differential of the Hochschild complex vanishes because $A$ is commutative.
This is the reason why the factor $H^n(X, \mathcal O_X)$ of $HH(X)$ is special.

\vskip 1pc

We consider deformations of a smooth variety $X$ over a field $k$ in this paper.
We fix an open covering $\{U_i\}_{i \in I}$ as above.

We consider the extension problem of NC deformations.
Let 
\begin{equation}\label{ext}
0 \to J \to R' \to R \to 0
\end{equation}
be an extension in $(\text{Art})$, where 
$J$ is an ideal such that $M'J = 0$.
We ask whether a given NC deformation $\mathcal A \in \text{Def}_X(R)$ extends to a deformation
$\mathcal A' \in \text{Def}_X(R')$ such that $\mathcal A = R \otimes_{R'} \mathcal A'$, and 
when it is extendible, then how is the set of extensions parametrized.

\begin{Thm}\label{T1T2}
Let $(R,M), (R',M') \in (\text{Art})$ and $J$ be as in (\ref{ext}), and
let $\mathcal A$ be an NC deformation of a smooth variety $X$ over $(R,M)$.
Then the following hold.

(1) There is a class $\xi^{(3,0)} \in J \otimes H^0(X, \bigwedge^3 T_X)$.
If $\xi^{(3,0)} = 0$, then there is a class $\xi^{(2,1)} \in J \otimes H^1(X, \bigwedge^2 T_X)$.
If $\xi^{(2,1)} = 0$, then there is a class $\xi^{(1,2)} \in J \otimes H^2(X, T_X)$.
And $\xi^{(3,0)} = \xi^{(2,1)} = \xi^{(1,2)} = 0$ hold if and only if there exists an NC deformation $\mathcal A'$
over $(R',M')$ which induces $\mathcal A$ over $(R,M)$.

(2) Assume that $\xi^{(3,0)} = \xi^{(2,1)} = \xi^{(1,2)} = 0$.
Then the set of isomorphism classes of all extensions $\mathcal A'$ 
corresponds bijectively to a $k$-vector space $J \otimes (H^0(X, \bigwedge^2 T_X) \oplus H^1(X, T_X))$.

(3) The obstruction classes $\xi^{(3,0)}, \xi^{(2,1)}, \xi^{(1,2)}$ are functorially determined in the following sense.
Let 
\begin{equation}\label{functorial}
\begin{CD}
0 @>>> J @>>> R' @>>> R @>>> 0 \\
@. @V{\beta_J}VV @V{\beta'}VV @V{\beta}VV \\
0 @>>> J_1 @>>> R'_1 @>>> R_1 @>>> 0
\end{CD}
\end{equation}
be a commutative diagram of extensions as in (\ref{ext}).
Let $\mathcal A \in \text{Def}_X(R)$ 
and let $\mathcal A^{(1)} = \text{Def}_X(\beta)(\mathcal A) \in \text{Def}_X(R_1)$.
Let $\xi^{(p,q)} \in J \otimes H^q(\bigwedge^p T_X)$ and 
$(\xi^{(1)})^{(p,q)} \in J_1 \otimes H^q(\bigwedge^p T_X)$ be the 
obstruction classes of extending $\mathcal A$ and $\mathcal A^{(1)}$ over $R'$ and $R'_1$, respectively.
Then $(\xi^{(1)})^{(p,q)} = \beta_J(\xi^{(p,q)})$ hold as long as they are defined.
\end{Thm}

\begin{proof}
(1) The flatness of $\mathcal A_i$ over $R$ implies that
$\mathcal A_i \cong R \otimes_k A_i$ when they are considered as $R$-modules.
Let us define $\mathcal A'_i = R' \otimes_k A_i$ as an $R'$-module.
We have $R \otimes_{R'} \mathcal A'_i \cong \mathcal A_i$ as $R$-modules.
In order to define $\mathcal A'$, we need to define multiplications on the $\mathcal A'_i$ and their gluing homomorphisms.

The $\mathcal A_i$ has an associative multiplication $\times_{R_i}$.
Let $\times_{R'_i}$ be any $R'_i$-bilinear multiplication on $\mathcal A'_i$ which induce $\times_{R_i}$.
We do not require that they are associative nor compatible with multiplications $\times_{R'_j}$ for other $j$.
For $x,y,z \in \mathcal A'_i$, let 
\[
f_i(\bar x, \bar y, \bar z) = (x \times_{R'_i} y) \times_{R'_i} z - x \times_{R'_i} (y \times_{R'_i} z) \in J \otimes_k A_i, 
\]
where $f_i$ depends only on the classes $\bar x, \bar y, \bar z \in A_i$ of $x,y,z$ modulo 
$M' \otimes_{R'} \mathcal A'_i$, because $M'J = 0$.
We also note that $J \otimes_{R'} \mathcal A'_i \cong J \otimes_k A_i$.
The multiplication $\times_{R'_i}$ is associative if and only if $f_i(\bar x, \bar y, \bar z) = 0$.
We can alternatively define $\tilde f_i: A_i^{\otimes 4} \to J \otimes A_i$ by
$\tilde f_i(x_0,x_1,x_2,x_3) = x_0f_i(x_1,x_2,x_3)$ as the Hochschild $3$-cochain.

Let $\phi_{ji} = \phi_{ji}^{\mathcal A}: \mathcal A_i \to \mathcal A_j$ be the gluing homomorphism of $\mathcal A$ for $i < j$.
They are compatible with multiplications and satisfy the compatibility condition 
$\phi_{kj} \circ \phi_{ji} = \phi_{ki}$ for $i < j < k$.
Let $\phi'_{ji} = \phi^{\mathcal A'}_{ji}: \mathcal A'_i \to \mathcal A'_j$ be any homomorphism as $R'$-modules 
which induces $\phi^{\mathcal A}_{ji}$ when the tensor product $R \otimes_{R'}$ is taken.
We set $\phi^{\mathcal A'}_{ii} = \text{Id}_{\mathcal A'_i}$ for $i = j$.
We do not require that the $\phi'_{ji}$ are compatible with multiplications nor satisfy the compatibility condition.

Let 
\[
g_{ji}(\bar x, \bar y) = \phi'_{ji}(x \times_{R'_i} y) - \phi'_{ji}(x) \times_{R'_j} \phi'_{ji}(y) 
\in J \otimes A_j
\]
for $x,y \in \mathcal A'_i$.
$g_{ji}$ depends only on the classes $\bar x, \bar y \in A_i$ as before.
The $\phi'_{ji}$ are compatible with the multiplications $\times_{R'_i}$ and $\times_{R'_j}$ 
if and only if $g_{ji}(\bar x, \bar y) = 0$.
We have $g_{ii} = 0$ for $i = j$.

Let 
\[
h_{kji}(\bar x) = \phi'_{kj} \circ \phi'_{ji}(x) - \phi'_{ki}(x) \in J \otimes A_k
\]
for $x \in \mathcal A'_i$.
$h_{kji}$ depends only on $\bar x \in A_i$, and the $\phi'_{ji}$ satisfy the compatibility condition
if and only if $h_{kji}(\bar x) = 0$.
We have $h_{kji} = 0$ if two among $i,j,k$ coincide.

\vskip 1pc

Now we calculate coboundaries of the Hochschild cochains $f_i$, $g_{ji}$ and $h_{kji}$ on $A_i$ 
with coefficients in $J \otimes A_i$, $J \otimes A_j$ and $J \otimes A_k$, respectively, 
with respect to the Hochschild boundaries.
We should be careful that different rings $A_i,A_j,A_k$ with $i < j < k$ are involved, and these rings are related by
homomorphisms $\phi^0 := \phi^A$. 

\begin{Lem}\label{df}
Let $i < j < k < l$, and let $f_i: A_i^{\otimes 3} \to J \otimes A_i$, 
$g_{ji}: A_i^{\otimes 2} \to J \otimes A_j$ and 
$h_{kji}: A^i \to J \otimes A_k$ be the Hochschild cochains defined above.
Then the Hochschild coboundaries satisfy the following equalities, where $\phi^0_{ji} := \phi^A_{ji}$, etc.

(1) $df_i = 0$.

(2) $dg_{ji}
= - f_j(\phi^0_{ji})^{\otimes 3} + \phi_{ji}^0f_i$.

(3) $dh_{kji}
= - g_{kj}(\phi^0_{ji})^{\otimes 2} + g_{ki} - \phi^0_{kj}g_{ji}$.

(4) $\phi^0_{lk}h_{kji} - h_{lji} + h_{lki} - h_{lkj}\phi^0_{ji} = 0$. 
\end{Lem}

\begin{proof}
(1) 
\[
\begin{split}
&df_i(\bar x, \bar y, \bar z, \bar w) \\
&= \bar xf_i(\bar y,\bar z,\bar w) - f_i(\bar x\bar y,\bar z,\bar w) + f_i(\bar x,\bar y\bar z,\bar w) 
- f_i(\bar x,\bar y,\bar z\bar w) + f_i(\bar x,\bar y,\bar z)\bar w \\
&= 0
\end{split}
\]
because
\[
\begin{split}
&x((yz)w-y(zw)) - (((xy)z)w - (xy)(zw)) + (x(yz))w - x((yz)w) \\
&- ((xy)(zw) - x(y(zw))) + ((xy)z-x(yz))w = 0
\end{split}
\]
where multiplications are $\times_{R'_i}$.

(2) 
\[
\begin{split}
&dg_{ji}(\bar x, \bar y, \bar z) \\
&= \phi^0_{ji}(\bar x)g_{ji}(\bar y, \bar z) - g_{ji}(\bar x \bar y, \bar z)
+ g_{ji}(\bar x, \bar y\bar z) - g_{ji}(\bar x, \bar y)\phi^0_{ji}(\bar z) \\
&= \phi'_{ji}(x) \times_{R'_j} \phi'_{ji}(y \times_{R'_i} z) - \phi'_{ji}(x) \times_{R'_j} (\phi'_{ji}(y) \times_{R'_j} \phi'_{ji}(z)) \\
&- \phi'_{ji}((x \times_{R'_i} y) \times_{R'_i} z) + \phi'_{ji}(x \times_{R'_i} y) \times_{R'_j} \phi'_{ji}(z) \\
&+ \phi'_{ji}(x \times_{R'_i} (y \times_{R'_i} z)) - \phi'_{ji}(x) \times_{R'_j} \phi'_{ji}(y \times_{R'_i} z) \\
&- \phi'_{ji}(x \times_{R'_i} y) \times_{R'_j} \phi'_{ji}(z) + (\phi'_{ji}(x) \times_{R'_j} \phi'_{ji}(y)) \times_{R'_j} \phi'_{ji}(z) \\
&= - \phi'_{ji}(x) \times_{R'_j} (\phi'_{ji}(y) \times_{R'_j} \phi'_{ji}(z)) + (\phi'_{ji}(x) \times_{R'_j} \phi'_{ji}(y)) \times_{R'_j} \phi'_{ji}(z) \\
&+ \phi'_{ji}(x \times_{R'_i} (y \times_{R'_i} z)) - \phi'_{ji}((x \times_{R'_i} y) \times_{R'_i} z) \\
&= - f_j(\phi^0_{ji}(\bar x), \phi^0_{ji}(\bar y), \phi_{ji}^0(\bar z)) + \phi_{ji}^0(f_i(\bar x, \bar y, \bar z))
\end{split}
\]
where we used $\overline{x \times_{R'_i} y} = \bar x\bar y$, etc.

(3)
\[
\begin{split}
&\phi'_{kj}\phi'_{ji}(x \times_{R'_i} y) = \phi'_{kj}(\phi'_{ji}(x) \times_{R'_j} \phi'_{ji}(y) + g_{ji}(\bar x,\bar y)) \\
&= \phi'_{kj}\phi'_{ji}(x) \times_{R'_k} \phi'_{kj}\phi'_{ji}(y) + g_{kj}(\phi^0_{ji}(\bar x),\phi^0_{ji}(\bar y)) 
+ \phi^0_{kj}g_{ji}(\bar x,\bar y) \\
&= (\phi'_{ki}(x) + h_{kji}(\bar x)) \times_{R'_k} (\phi'_{ki}(y) + h_{kji}(\bar y)) \\
&+ g_{kj}(\phi^0_{ji}(\bar x),\phi^0_{ji}(\bar y)) + \phi^0_{kj}g_{ji}(\bar x,\bar y) \\
&= \phi'_{ki}(x \times_{R'_i} y) - g_{ki}(\bar x,\bar y) + \phi^0_{ki}(\bar x)h_{kji}(\bar y) 
+ h_{kji}(\bar x) \phi^0_{ki}(\bar y) \\
&+ g_{kj}(\phi^0_{ji}(\bar x),\phi^0_{ji}(\bar y)) + \phi^0_{kj}g_{ji}(\bar x,\bar y).
\end{split}
\]
Hence
\[
\begin{split}
&dh_{kji}(\bar x,\bar y) = \phi^0_{ki}(\bar x) h_{kji}(\bar y) - h_{kji}(\bar x \bar y) + h_{kji}(\bar x)\phi^0_{ki}(\bar y) \\
&= - g_{kj}(\phi^0_{ji}(\bar x),\phi^0_{ji}(\bar y)) + g_{ki}(\bar x,\bar y) - \phi^0_{kj}g_{ji}(\bar x,\bar y).
\end{split}
\]

(4) 
\[
\begin{split}
&\phi^0_{lk}h_{kji}(\bar x) - h_{lji}(\bar x) + h_{lki}(\bar x) - h_{lkj}(\phi^0_{ji}(\bar x)) \\
&= \phi'_{lk}(\phi'_{kj}\phi'_{ji}(x) - \phi'_{ki}(x)) - (\phi'_{lj}\phi'_{ji}(x) - \phi'_{li}(x)) \\
&+ (\phi'_{lk}\phi'_{ki}(x) - \phi'_{li}(x)) - (\phi'_{lk}\phi'_{kj}\phi'_{ji}(x) - \phi'_{lj}\phi'_{ji}(x)) \\
&= 0.
\end{split}
\]
\end{proof}

A different choice of multiplication on $\mathcal A'_i$ can be written as
\[
(x, y) \mapsto x \times_{R'_i} y + b_i(\bar x, \bar y)
\]
for $x,y \in \mathcal A'_i$, where $b_i(\bar x, \bar y) \in J \otimes_{R'} \mathcal A'_i = J \otimes_k A_i$ depends
only on $\bar x, \bar y \in A_i$, since $M'J = 0$. 
At the same time, a different gluing of the $\mathcal A'_i$ can be written as
\[
x \mapsto \phi'_{ji}(x) + c_{ji}(\bar x)
\]
for $x \in \mathcal A'_i$, where $c_{ji}(\bar x) \in J \otimes_{R'} \mathcal A'_j = J \otimes_k A_j$.

\begin{Lem}\label{f'}
Assume that $f_i, g_{ji}, h_{kji}$ are changed to $f'_i, g'_{ji}, h'_{kji}$
under these new multiplications and gluing.
Then the following hold.

(1) $f'_i = f_i - db_i$.

(2) $g'_{ji} = g_{ji} + \phi^0_{ji}b_i - b_j(\phi^0_{ji})^{\otimes 2} - dc_{ji}$.

(3) $h'_{kji} = h_{kji} + \phi^0_{kj}c_{ji} - c_{ki} + c_{kj}\phi^0_{ji}$.
\end{Lem}

\begin{proof}
(1) 
\[
\begin{split}
&f'_i(\bar x, \bar y, \bar z) \\
&= (x \times_{R'_i} y + b_i(\bar x, \bar y)) \times_{R'_i} z + b_i(\bar x\bar y, \bar z) \\
&- x \times_{R'_i} (y \times_{R'_i} z + b_i(\bar y, \bar z)) - b_i(\bar x, \bar y \bar z) \\
&= f_i(\bar x, \bar y, \bar z) - \bar x b_i(\bar y, \bar z) + b_i(\bar x\bar y, \bar z) 
 - b_i(\bar x, \bar y \bar z) + b_i(\bar x, \bar y) \bar z \\
&= f_i(\bar x, \bar y, \bar z) - db_i(\bar x, \bar y, \bar z).
\end{split}
\]

(2) 
\[
\begin{split}
&g'_{ji}(\bar x, \bar y) \\
&= \phi'_{ji}(x \times_{R'_i} y + b_i(\bar x, \bar y)) + c_{ji}(\bar x \bar y) \\
&- ((\phi'_{ji}(x) + c_{ji}(\bar x)) \times_{R'_j} (\phi'_{ji}(y) + c_{ji}(\bar y)) + b_j(\phi^0_{ji}(\bar x), \phi^0_{ji}(\bar y)) \\
&= g_{ji}(\bar x, \bar y) + \phi^0_{ji}b_i(\bar x, \bar y) - b_j(\phi^0_{ji}(\bar x), \phi^0_{ji}(\bar y)) \\
&- \phi^0_{ji}(\bar x) c_{ji}(\bar y) + c_{ji}(\bar x \bar y) - c_{ji}(\bar x) \phi^0_{ji}(\bar y) \\
&= g_{ji}(\bar x, \bar y) + \phi^0_{ji}b_i(\bar x, \bar y) - b_j(\phi^0_{ji}(\bar x), \phi^0_{ji}(\bar y))
- dc_{ji}(\bar x, \bar y).
\end{split}
\]

(3) 
\[
\begin{split}
&h'_{kji}(\bar x) = \phi'_{kj}(\phi'_{ji}(x) + c_{ji}(\bar x)) + c_{kj}(\phi^0_{ji}(\bar x)) - (\phi'_{ki}(x) + c_{ki}(\bar x)) \\
&= h_{kji}(\bar x) + \phi^0_{kj}(c_{ji}(\bar x)) - c_{ki}(\bar x) + c_{kj}(\phi^0_{ji}(\bar x)).
\end{split}
\]
\end{proof}

By Lemma \ref{df} (1) and Lemma \ref{f'} (1), the Hochschild cochain $f_i$ is closed and is 
determined up to the Hochschild coboundary $db_i$.
Thus the cohomology class $v(f_i):= [f_i] \in J \otimes \Gamma(\bigwedge^3 T_{U_i})$ is well-defined
in the sense that it is independent of the choice of the multiplication on $\mathcal A'_i$.
Moreover, Lemma \ref{df} (2) 
implies that the $v(f_i)$ glue together to define a global cohomology class
$\xi^{(3,0)} = v(f) \in J \otimes \Gamma(\bigwedge^3 T_X)$.
$\xi^{(3,0)} = 0$ is equivalent to saying that $v(f_i) = 0$ for all $i$, and in turn to that $f'_i = 0$ for 
suitable choices of the $b_i$.
Thus we have $\xi^{(3,0)} = 0$ if and only if there are choices of associative multiplications on the $\mathcal A'_i$
extending those on the $\mathcal A_i$.

\vskip 1pc

Now we assume that $\xi^{(3,0)} = 0 \in J \otimes \Gamma(\bigwedge^3 T_X)$. 
Thus we assume that $f'_i = 0$ for all $i$.
We fix one set of the $b_i$ which determine associative multiplications $\times_{R'_i}$ on the $\mathcal A'_i$.
The existence of such $b_i$ is guaranteed by the assumption that $\xi^{(3,0)} = 0$.
Then we may assume that $f_i = 0$ for all $i$.
It follows that $db_i = 0$ for all $i$, and
the Hochschild $2$-cochain $g_{ji}: A_i^{\otimes 2} \to J \otimes A_j$ of $A_i$ with coefficients in $J \otimes A_j$ 
satisfies the Hochschild cocyle condition that $dg_{ji} = 0$ by Lemma \ref{df} (2).
We note that there are still choices of the $b_i$ as long as $db_i = 0$.

Thus $g_{ji}$ for $i < j$ and $b_i$ determine $2$-vectors $v(g_{ji}) \in J \otimes \Gamma(\bigwedge^2 T_{U_j})$
and $v(b_i) \in J \otimes \Gamma(\bigwedge^2 T_{U_i})$.
From Lemma \ref{df} (3), the \v Cech $1$-cochain $v(g_{ji})$ satisfies the cocycle condition that 
$v(g_{ki}) = v(g_{kj}) + v(g_{ji}) \vert_{U_k}$.
Moreover by Lemma \ref{f'} (2), the \v Cech cocycles $v(g_{ji})$ and $v(g'_{ji})$ are equivalent up to the \v Cech coboundary of 
the \v Cech $0$-cochain $v(b_i)$.

Now we define a \v Cech cocycle $v(g_{ji})$ for the general case of non-ordered pairs $i,j$.
We set $l = \max \{i,j\}$ corresponding to
$U_i \cap U_j = U_l$, and define
\[
v(g_{ji}) = v(g_{li}) - v(g_{lj}) \in J \otimes \Gamma(\bigwedge^2 T_{U_l}).
\]
We note that, if $i < j$, then $l = j$, and $g_{lj} = 0$, hence $v(g_{ji}) = v(g_{li})$ as expected.

We check the cocycle condition for unordered indexes $i,j,k$.
Let $l = \max \{i,j\}$, $m = \max \{j,k\}$, $n = \max \{i,k\}$ and $t = \max \{i,j,k\}$.
Then
\[
\begin{split}
&(v(g_{ji}) - v(g_{ki}) + v(g_{kj}))\vert_{U_t} \\
&= (v(g_{li}) - v(g_{lj}) - v(g_{ni}) + v(g_{nk}) + v(g_{mj}) - v(g_{mk}))\vert_{U_t} \\
&= v(g_{ti}) - v(g_{tj}) - v(g_{ti}) + v(g_{tk}) + v(g_{tj}) - v(g_{tk}) = 0.
\end{split}
\]
We also have
\[
\begin{split}
&v(g_{ji}) - v(g'_{ji}) = v(g_{li}) - v(g_{lj}) - v(g'_{li}) + v(g'_{lj}) \\
&= (v(b_l) - v(b_i) - v(b_l) + v(b_j)) \vert_{U_l} = (v(b_j) - v(b_i))\vert_{U_l}.
\end{split}
\]
Therefore the collection $\{v(g_{ji})\}$ defines a \v Cech cohomology class
$\xi^{(2,1)} := v(g) \in J \otimes \check H^1(\bigwedge^2 T_X)$.
We have $\xi^{(2,1)} = 0$ if and only if there are suitable choices of the $b_i$ and the $c_{ji}$ such that 
$g'_{ji} = 0$.
Thus $\xi^{(2,1)} = 0$ holds if and only if there is a gluing of the $\mathcal A_i$ which is compatible with the multiplications.

\vskip 1pc

Now we assume that $\xi^{(2,1)} = 0 \in J \otimes \check H^1(\bigwedge^2 T_X)$.  
Then we may assume that $g_{ji} = 0$ and $dc_{ji} = 0$ for all $i < j$.
We note that we have still choices of the $c_{ji}$ as long as $dc_{ji} = 0$.

By Lemma \ref{df} (3), the Hochschild $1$-cochain $h_{kji}$ is closed: $dh_{kji} = 0$.
We define $v(h_{kji}) := [h_{kji}] \in J \otimes \Gamma(T_{U_k})$.
We note that the Hochschild coboundary map is trivial for $0$-cochains, because $A_i$ is commutative. 
Therefore $v(h_{kji}) = 0$ is equivalent to that $h_{kji} = 0$.

By Lemma \ref{df} (4), the \v Cech $2$-cocycle $v(h_{kji})$ is closed in the sense that
\[
(h_{kji} - h_{lji} + h_{lki} - h_{lkj}) \vert_{U_l} = 0
\]
for $i < j < k < l$.

\vskip 1pc

In general, we have the following lemma:

\begin{Lem}\label{S_n}
Let $I$ be a poset as above, $n$ a positive integer and $E$ an abelian group.
Let $I_n = \{(i_n, \dots, i_1) \in I^n \mid i_1 < \dots < i_n\}$ be the subset of ordered indexes of length $n$ in the poset $I$. 
Assume that there is a map $h: I_n \to E$ satisfying the cocycle condition: 
\[
\sum_{p=0}^n (-1)^p h(i_n, \dots, i_{p+1}, i_{p-1}, \dots, i_0) = 0
\]
for $i_0 < \dots < i_n$.
Set $h(i_n,\dots,i_1) = 0$ if any two among the $i_p$ coincide.
Define a map $\tilde h: I^n \to E$ by the following formula:
\[
\tilde h(i_n,\dots,i_1) = \sum_{s \in S_n} \text{sgn}(s) h(j_{s,n},\dots,j_{s,1})
\]
where 
\[
j_{s,k} = \max \{i_{s(1)}, \dots, i_{s(k)}\}.
\] 
Then $\tilde h \vert_{\bar I_n} = h$ for $\bar I_n = \{(i_n, \dots, i_1) \in I^n \mid i_1 \le \dots \le i_n\}$, 
and $\tilde h$ still satisfies the cocycle condition
\[
\sum_{p=0}^n (-1)^p \tilde h(i_n, \dots, i_{p+1}, i_{p-1}, \dots, i_0) = 0
\]
for arbitrary indexes $i_0, \dots, i_n$.
\end{Lem}

\begin{proof}
It is clear that $\tilde h \vert_{\bar I_n} = h$ from the definition of $\tilde h$.
We take a sequence of arbitrary indexes $(i_n, \dots, i_0) \in I^{n+1}$.
Let the permutation group $S_{n+1}$ act on the set $\{0,1,\dots,n\}$. 
For each $s \in S_{n+1}$, we define an increasing sequence $j_{s,0} \le j_{s,1} \le \dots \le j_{s,n}$ by
\[
j_{s,k} = \max \{i_{s(0)}, \dots, i_{s(k)}\}.
\]
We note that $j_{s,n} = j_n = \max \{i_0, \dots, i_n\}$ is independent of $s$, and
$j_{s,n-1}$ has $n+1$ possibilities, with possibly counting with multiplicities, depending on $s(n)$.

Since $h$ satisfies the cocycle condition, we have 
\[
\sum_{q=0}^n (-1)^q h(j_{s,n}, \dots, j_{s, q+1}, j_{s, q-1}, \dots, j_{s,0}) = 0.
\]
Thus, for each $p$, we have 
\[
\begin{split}
&\tilde h(i_n, \dots, i_{p+1}, i_{p-1}, \dots, i_0)
= \sum_{s(n) = p} (-1)^{n-p} \, \text{sgn}(s) h(j_{s,n-1}, \dots, j_{s,0}) \\
&= \sum_{s(n) = p} (-1)^{n-p} \, \text{sgn}(s) \sum_{q=0}^{n-1} (-1)^{n-1-q} h(j_{s,n}, \dots, j_{s, q+1}, j_{s, q-1}, \dots, j_{s,0}).
\end{split}
\]

We claim that, for a given $s$ and $q < n$, there is a uniquely determined $s' \in S_{n+1}$ such that 
$s \ne s'$ and that $\{s(0),\dots, s(p)\} = \{s'(0), \dots, s'(p)\}$ for $p \ne q$.
Indeed $s'$ is given by $s(p) = s'(p)$ for $p \ne q, q+1$, $s(q) = s'(q+1)$ and $s(q+1) = s'(q)$.
We note that $\text{sgn}(s) = - \text{sgn}(s')$, and we have
\[
h(j_{s,n}, \dots, j_{s, q+1}, j_{s, q-1}, \dots, j_{s,0}) = h(j_{s',n}, \dots, j_{s', q+1}, j_{s', q-1}, \dots, j_{s',0}).
\]
Now we calculate
\[
\begin{split}
&\sum_{p=0}^n (-1)^p \tilde h(i_n, \dots, i_{p+1}, i_{p-1}, \dots, i_0) \\
&= \sum_{p=0}^n \sum_{s(n) = p} \text{sgn}(s) \sum_{q=0}^{n-1} (-1)^{q+1} h(j_{s,n}, \dots, j_{s, q+1}, j_{s, q-1}, \dots, j_{s,0}) \\
&= 0
\end{split}
\]
since the sum is over the pairs of $s$ and $s'$ with opposite signs. 
\end{proof}

By Lemma \ref{S_n}, we extend the definition of $v(h_{kji})$ to the unordered $i,j,k$.
Therefore there is a well-defined \v Cech cocycle $\xi^{(1,2)} :=[v(h)] = \{v(h_{kji})\} \in J \otimes \check H^2(T_X)$.

\begin{Lem}
$\xi^{(1,2)} = 0$ if and only if there is a suitable choice of the $c_{ji}$ such that 
$h'_{kji} = 0$ for $i < j < k$.
\end{Lem}

Thus $\xi^{(1,2)} = 0$ holds if and only if there is a gluing of the $\mathcal A_i$ which satisfies the \v Cech cocycle condition.

\begin{proof}
Assume that $\xi^{(1,2)} = 0$.
Then there is a \v Cech cochain $v(c_{ji}) \in J \otimes \Gamma(T_{U_i \cap U_j})$ such that 
$v(h_{kji}) = (v(c_{ji}) - v(c_{ki}) + v(c_{ji})) \vert_{U_i \cap U_j \cap U_k}$.
In particular, there are such $v(c_{ji})$ for $i < j$.
By Lemma \ref{f'} (3), there exists gluing of the $\mathcal A'_i$ such that $h'_{kji} = 0$.

Conversely, assume that there are $c_{ji}$ for $i < j$ such that $h'_{kji} = 0$.
We define $c_{ji} = c_{ki} - c_{kj}$ for unordered $i,j$, where we set $k = \max \{i,j\}$.
Then, omitting $v$ and $\vert_{U_t}$, we have
\[
\begin{split}
&h_{kji} = h_{tli} - h_{tlj} + h_{tmj} - h_{tmk} + h_{tnk} - h_{tni} \\
&= c_{li} - c_{ti} + c_{tl} - (c_{lj} - c_{tj} + c_{tl}) 
+ c_{mj} - c_{tj} + c_{tm} - (c_{mk} - c_{tk} + c_{tm}) \\
&+ c_{nk} - c_{tk} + c_{tn} - (c_{ni} - c_{ti} + c_{tn}) \\
&= c_{li} - c_{lj} + c_{mj} - c_{mk} + c_{nk} - c_{ni} 
= c_{ji} - c_{ki} + c_{kj}.
\end{split}
\]
Hence $\xi^{(1,2)} = 0$.
\end{proof}

The combination of the above argument yields the proof of (1).

\vskip 1pc

(2) We assume that $\xi^{(3,0)} = \xi^{(2,1)} = \xi^{(1,2)} = 0$.
The argument in the proof of (1) showed that the set of choices of the 
multiplications and gluing for which $f_i = g_{ij} = h_{ijk} = 0$ hold corresponds bijectively to the set of the $b_i$ and $c_{ji}$ 
which satisfy the following conditions 
\[
\begin{split}
&db_i = 0, \,\,\, dc_{ji} = 0, \\
&\phi^0_{ji}b_i(\bar x, \bar y) - b_j(\phi^0_{ji}(\bar x), \phi^0_{ji}(\bar y)) = 0, \\
&\phi^0_{kj}(c_{ji}(\bar x)) - c_{ki}(\bar x) + c_{kj}(\phi^0_{ji}(\bar x)) = 0
\end{split}
\]
where we have $i < j < k$.
Since $db_i = dc_{ji} = 0$, the $b_i$ and the $c_{ji}$ determine classes 
$v(b_i) \in J \otimes \Gamma(\bigwedge^2 T_{U_i})$ and 
$v(c_{ji}) \in J \otimes \Gamma(T_{U_j})$.
They satisfy the cocycle condition, and we obtain
$v(b) \in J \otimes H^0(\bigwedge^2 T_X)$ and $v(c) \in J \otimes \check H^1(T_X)$, where 
we define $v(c_{ji}) = v(c_{ki}) - v(c_{kj})$ with $k = \max \{i,j\}$ for unordered $i,j$.

We still need to take into account of the fact that some of the deformations $\mathcal A'$ of 
$\mathcal A$ thus obtained are isomorphic, 
and we need to take the quotient set of the set of the $b_i, c_{ji}$.
Let $\tilde b_i$ and $\tilde c_{ji}$ be other choices for multiplications and gluing.
We will prove that $v(b) = v(\tilde b) \in J \otimes H^0(\bigwedge^2 T_X)$ and 
$v(c) = v(\tilde c) \in J \otimes \check H^1(T_X)$ hold if and only if 
the resulting NC deformations $\mathcal A'$ and $\tilde{\mathcal A}'$ are isomorphic.

An isomorphism $\mathcal A' \to \tilde{\mathcal A}'$ is given by 
isomorphisms $\mathcal A'_i \to \tilde{\mathcal A}'_i$ of $R'$-modules which induce identity over $R$ and 
are compatible with multiplications and gluing.
We can write them as
\[
x \mapsto x + e_i(\bar x)
\]
for $x \in \mathcal A'_i$, where $e_i(\bar x) \in J \otimes A_i$.

They are compatible with the multiplications if and only if 
\[x \times_{R'_i} y + b_i(\bar x, \bar y) + e_i(\bar x \bar y)
= (x + e_i(\bar x)) \times_{R'_i} (y + e_i(\bar y)) + \tilde b_i(\bar x, \bar y), 
\]
that is 
\begin{equation}\label{e1}
\tilde b_i(\bar x, \bar y) - b_i(\bar x, \bar y) = - \bar x e_i(\bar y) + e_i(\bar x \bar y) - e_i(\bar x) \bar y
= - de_i(\bar x, \bar y).
\end{equation}
They are compatible with the gluing if and only if 
\[
\phi'_{ji}(x) + c_{ji}(\bar x) + \phi^0_{ji}(e_i(\bar x)) = \phi'_{ji}(x) + \tilde c_{ji}(\bar x) + e_j(\bar x), 
\]
that is 
\begin{equation}\label{e2}
\tilde c_{ji}(\bar x) - c_{ji}(\bar x) = - e_j(\bar x) + \phi^0_{ji}(e_i(\bar x)).
\end{equation}

There exist such $e_i$'s which sends $(b_i, c_{ji})$ to $(\tilde b_i,\tilde c_{ji})$ through (\ref{e1}) and (\ref{e2}) 
if and only if 
$v(b) = v(\tilde b) \in J \otimes H^0(\bigwedge^2 T_X)$ and 
$v(c) = v(\tilde c) \in J \otimes \check H^1(T_X)$. 
Indeed the reduction to the Hochschild cohomology class of the $b_i$ is done by the boundaries $de_i$, and the remaining 
reduction to the \v Cech cohomology class of the $c_{ji}$ is done by those $e_i$ such that $de_i = 0$. 
Thus the part (2) is proved.

\vskip 1pc

(3) Let $\times_{R'_i}$ and $\phi'_{ij}$ be a choice of multiplications and gluing over $R'$, and let 
$\xi^{(p,q)} \in J \otimes H^q(X, \bigwedge^p T_X)$ be the corresponding obstruction classes determined 
by the cochains $f_i,g_{ji},h_{kji}$.
Then the map $\beta'$ induces a choice of multiplications and gluing $\times_{R'_{1,i}}$ and $\phi'_{1, ij}$ 
over $R'_1$.
The corresponding cochains satisfy $(f_{1,i}, g_{1,ji}, h_{1,kji}) = \beta'(f_i, g_{ji}, h_{kji})$.
Hence $(\xi^{(1)})^{(p,q)} = \beta_J(\xi^{(p,q)})$.
\end{proof}

Now we prove the existence of a semi-universal NC deformation of $X$ using the Schlessinger conditions 
(\cite{Schlessinger}):

\begin{Prop}\label{Schlessinger condiiton}
Let $X$ be a smooth algebraic variety, and
assume that $\dim H^0(X, \bigwedge^2 T_X) + \dim H^1(X, T_X)$ is finite.
Then there exists a semi-universal deformation for the NC deformation functor $\text{Def}_X$.
\end{Prop}

\begin{proof}
We need to check the conditions ($\tilde H$) and ($H_e$) of \cite{Sernesi} Chapter 2 (cf. \cite{NC base} Theorem 4.1).
($\tilde H$) says that a natural map
\[
\text{Def}_X(R_1 \times_{R_0} R_2) \to \text{Def}_X(R_1) \times_{\text{Def}_X(R_0)} \text{Def}_X(R_2)
\]
is surjective for morphisms $R_i \to R_0$ in $(\text{Art})$, while ($H_e$) says that it is bijective 
when $R_0 = k$ and $R_1 = k[\epsilon]/(\epsilon^2)$.
($\tilde H$) is proved in the following lemma.
This map may be non-injective because an automorphism of a deformation over $R_0$ may be 
non-extendable to a deformation over $R_1$.
If $R_0 = k$, then there is no automorphism by definition, hence ($H_e$) holds.
\end{proof}

\begin{Lem}\label{glue}
Let $P = k[[t_1,\dots,t_n]]$, $R_i = P/I_i \in (\text{Art})$ for $i = 1,2$, and let 
$R_0 = P/(I_1+I_2)$ and $R' = P/(I_1 \cap I_2)$. 
Assume that there are NC deformations $\mathcal A^{(i)}$ of 
$X$ over $R_i$ such that $R_0 \otimes_{R_1} \mathcal A^{(1)} \cong R_0 \otimes_{R_2} \mathcal A^{(2)} 
:= \mathcal A^{(0)}$.
Then there exists an NC deformation $\mathcal A'$ of $X$ over $R'$ such that  
$R_i \otimes_{R'} \mathcal A' \cong \mathcal A^{(i)}$.
\end{Lem}

\begin{proof}
We claim that $P/I_1 \times_{P/(I_1+I_2)} P/I_2 \cong P/(I_1 \cap I_2)$.
Indeed there is a natural map from the right hand side to the left hand side which sends $a \mod I_1 \cap I_2$ to 
$(a \mod I_1, a \mod I_2)$.
It is injective, because $a \in I_1$ and $a \in I_2$ implies that $a \in I_1 \cap I_2$.

We prove that it is surjective.
Suppose that $a_i + I_i$ for $a_i \in P$ give the same class $a_1 + I_1 + I_2 = a_2 + I_1 + I_2$.
Then there are elements $b_i, b_i' \in I_i$ such that 
$a_1 + b_1 + b_2 = a_2 + b'_1 + b'_2$.
Thus $a_1 + b_1-b'_1 = a_2 + b'_2 - b_2 = a \in P$.
Then $a \mod I_1 \cap I_2$ is mapped to  
$(a_1 \mod I_1, a_2 \mod I_2)$.

Similarly, we obtain 
\[
(P/I_1 \otimes A_i) \times_{P/(I_1+I_2) \otimes A_i} (P/I_2 \otimes A_i) 
\cong P/(I_1 \cap I_2) \otimes A_i
\]
as $P$-modules.
Then $\mathcal A' = \mathcal A^{(1)} \times_{\mathcal A^{(0)}} \mathcal A^{(2)}$ with ring structures and gluing
is the desired NC deformation.
\end{proof}

As a corollary of Theorem \ref{T1T2} with Proposition \ref{Schlessinger condiiton}, we obtain the following:

\begin{Thm}\label{semi-universal}
Let $X$ be a smooth algebraic variety, and
assume that 
\[
\begin{split}
&n = \dim H^0(X, \bigwedge^2 T_X) + \dim H^1(X, T_X), \\
&m = \dim H^0(X, \bigwedge^3 T_X) + \dim H^1(X, \bigwedge^2 T_X) + \dim H^2(X, T_X)
\end{split}
\]
are finite.
Then there exists a semi-universal deformation for the NC deformation functor $\text{Def}_X$ over 
\[
R := k[[t_1,\dots,t_n]]/(f_1,\dots,f_m)
\]
for some formal power series $f_i$ of orders in $[2, \infty]$.
\end{Thm}

\section{twisted NC deformations}

We consider twisted version of NC schemes and quasi-coherent sheaves on them.
We refer to \cite{DVLL} for definitions of more general situations.

\begin{Defn}
A {\em twisted NC scheme} $(\mathcal A, \tau)$ and a 
{\em $\tau$-twisted quasi-coherent $\mathcal A$-module} $\mathcal M$ 
consist of the following data:

\begin{enumerate}
\item A poset $I$ on which the maximum $\max\{i,j\}$ exists for any two elements $i,j \in I$;

\item A set of associative algebras $\mathcal A_i$ parametrized by $i \in I$.

\item Algebra homomorphisms $\phi_{ji}: \mathcal A_i \to \mathcal A_j$ for 
any pair $i < j$
(we set $\phi_{ii} = \text{Id}_{\mathcal A_i}$);

\item A set of right $\mathcal A_i$-modules $\mathcal M_i$.

\item Right module isomorphisms $\psi_{ji}: \mathcal M_i \otimes_{\mathcal A_i} \mathcal A_j \to \mathcal M_j$
over $\phi_{ji}$, i.e., $\psi_{ji}(mx) = \psi_{ji}(m) \phi_{ji}(x)$ for $m \in \mathcal M_i$ and $x \in \mathcal A_i$
(we set $\psi_{ii} = \text{Id}_{\mathcal M_i}$);

\item A set of invertible elements $\tau_{kji} \in \mathcal A_k^{\times}$ for $i < j < k$
(we set $\tau_{kji} = 1$ if two among $i,j,k$ coincide).
\end{enumerate}

which satisfy the following conditions:

\begin{itemize}
\item $\tau_{lji}\tau_{lkj} = \tau_{lki}\phi_{lk}(\tau_{kji})$ for $i < j < k < l$.

\item $\phi_{kj}\phi_{ji}(x) = \tau_{kji}^{-1}\phi_{ki}(x)\tau_{kji}$ for $i < j < k$ and $x \in \mathcal A_i$.

\item $\psi_{kj}\psi_{ji}(m) = \psi_{ki}(m)\tau_{kji}$ for $i < j < k$ and $m \in \mathcal M_i$.

\item $\phi_{ji}$ is two-sided flat, and the natural map $\mathcal A_j \to \mathcal A_j \otimes_{\mathcal A_i} \mathcal A_j$ 
is bijective.
\end{itemize}
\end{Defn}

The first condition is required by the following comparison for $m \in \mathcal M_i$:
\[
\begin{split}
&\psi_{lk}\psi_{kj}\psi_{ji}(m) = \psi_{lj}(\psi_{ji}(m))\tau_{lkj} = \psi_{li}(m)\tau_{lji}\tau_{lkj}, \\
&\psi_{lk}\psi_{kj}\psi_{ji}(m) = \psi_{lk}(\psi_{ki}(m)\tau_{kji}) = \psi_{li}(m)\tau_{lki}\phi_{lk}(\tau_{kji}).
\end{split}
\]
The second condition is reduced to the usual cocycle condition if $\tau_{kji}$ is in the center of $\mathcal A_k$.
The meaning of the third condition is that $\mathcal M$ is twisted by $\tau$.
The second and third conditions are compatible by the following calculation for $m \in \mathcal M_i$ and $x \in \mathcal A_i$:
\[
\begin{split}
&\psi_{kj}\psi_{ji}(mx) = \psi_{kj}(\psi_{ji}(m)\phi_{ji}(x)) = \psi_{kj}\psi_{ji}(m) \phi_{kj}\phi_{ji}(x) \\
&= \psi_{ki}(m)\tau_{kji} \tau_{kji}^{-1}\phi_{ki}(x)\tau_{kji} = \psi_{ki}(mx)\tau_{kji}.
\end{split}
\]
The last condition is automatic for NC deformations by Proposition \ref{flat birational}.

$\mathcal M$ is said to be {\em coherent} if the $\mathcal M_i$ are finitely generated $\mathcal A_i$-modules.
The category of $\tau$-twisted coherent (resp. quasi-coherent) $\mathcal A$-modules is denoted by 
$\text{coh}(\mathcal A, \tau)$ (resp. $\text{Qcoh}(\mathcal A, \tau)$).
These categories are behind the definitions on the twisted NC deformations, but they do not appear logically in the arguments.
They are abelian categories thanks to the flatness of the gluing $\phi_{ji}$.

\vskip 1pc

We define an equivalence between twisted NC schemes:

\begin{Defn}
An {\em equivalence} between twisted schemes, written as $(\mathcal A, \tau) \sim (\mathcal A', \tau')$, 
consists of algebra isomorphisms $\epsilon_i: \mathcal A_i \to \mathcal A'_i$ and invertible elements
$\rho_{ji} \in \mathcal A_j$ which satisfy the following conditions:

\begin{enumerate}

\item $\phi'_{ji}(\epsilon_i(x)) = \epsilon_j(\rho_{ji}^{-1}\phi_{ji}(x)\rho_{ji})$.

\item $\tau'_{kji} = \epsilon_k(\rho_{ki}^{-1} \tau_{kji} \phi_{kj}(\rho_{ji}) \rho_{kj})$.

\end{enumerate}

\end{Defn}

\begin{Rem}
(1) If $(\mathcal A, \tau) \sim (\mathcal A', \tau')$, then there is an equivalence of abelian categories
\[
\Phi: \text{Qcoh}(\mathcal A, \tau) \to \text{Qcoh}(\mathcal A', \tau')
\]
given by isomorphisms $\epsilon^{\mathcal M}_i: \mathcal M_i \to \Phi(\mathcal M)_i$ over $\epsilon_i$
with gluing 
\[
\psi^{\Phi(\mathcal M)}_{ji}(\epsilon^{\mathcal M}_i(m)) 
= \epsilon^{\mathcal M}_j (\psi^{\mathcal M}_{ji}(m) \rho_{ji})
\]
for $m \in \mathcal M_i$.

We check first that $\psi^{\mathcal M}_{ji}$ and $\psi^{\Phi(\mathcal M)}_{ji}$ are compatible with 
$\phi_{ji}$ and $\phi'_{ji}$.
Indeed we have
\[
\begin{split}
&\psi^{\Phi(\mathcal M)}_{ji}(\epsilon^{\mathcal M}_i(mx)) 
= \epsilon^{\mathcal M}_j (\psi^{\mathcal M}_{ji}(mx) \rho_{ji})
= \epsilon^{\mathcal M}_j (\psi^{\mathcal M}_{ji}(m) \phi_{ji}(x) \rho_{ji}) \\
&= \epsilon^{\mathcal M}_j (\psi^{\mathcal M}_{ji}(m) \rho_{ji}) \phi'_{ji}(\epsilon_i(x)) 
= \psi^{\Phi(\mathcal M)}_{ji}(\epsilon^{\mathcal M}_i(m))\phi'_{ji}(\epsilon_i(x))
\end{split}
\]
for $m \in \mathcal M_i$ and $x \in \mathcal A_i$.
We check the compatibility with the change of twists:
\[
\begin{split}
&\epsilon^{\mathcal M}_k(\psi^{\mathcal M}_{ki}(m) \rho_{ki}) \tau'_{kji} 
=\psi^{\Phi(\mathcal M)}_{ki}(\epsilon^{\mathcal M}_i(m)) \tau'_{kji} 
= \psi^{\Phi(\mathcal M)}_{kj}\psi^{\Phi(\mathcal M)}_{ji}(\epsilon^{\mathcal M}_i(m)) \\
&= \psi^{\Phi(\mathcal M)}_{kj}\epsilon^{\mathcal M}_j(\psi^{\mathcal M}_{ji}(m) \rho_{ji})
= \epsilon^{\mathcal M}_k (\psi^{\mathcal M}_{kj}(\psi^{\mathcal M}_{ji}(m)) \phi_{kj}(\rho_{ji})\rho_{kj}) \\
&= \epsilon^{\mathcal M}_k (\psi^{\mathcal M}_{ki}(m) \tau_{kji} \phi_{kj}(\rho_{ji})\rho_{kj})
\end{split}
\]
We can also check that the compatibility conditions on $\tau$ and $\tau'$ are equivalent:
\[
\begin{split}
&\tau'_{lji}\tau'_{lkj} = \epsilon_l(\rho_{li}^{-1} \tau_{lji} \phi_{lj}(\rho_{ji}) \rho_{lj}
\rho_{lj}^{-1} \tau_{lkj} \phi_{lk}(\rho_{kj}) \rho_{lk}) \\
&= \epsilon_l(\rho_{li}^{-1} \tau_{lji} \phi_{lj}(\rho_{ji}) \tau_{lkj} \phi_{lk}(\rho_{kj}) \rho_{lk}) \\
&= \epsilon_l(\rho_{li}^{-1} \tau_{lji} \tau_{lkj} \phi_{lk}\phi_{kj}(\rho_{ji})\phi_{lk}(\rho_{kj}) \rho_{lk}) \\
&= \epsilon_l(\rho_{li}^{-1} \tau_{lki} \phi_{lk}(\tau_{kji}) \phi_{lk}\phi_{kj}(\rho_{ji})\phi_{lk}(\rho_{kj}) \rho_{lk}) \\
&= \epsilon_l(\rho_{li}^{-1} \tau_{lki} \phi_{lk}(\rho_{ki}) \rho_{lk}
\cdot \rho_{lk}^{-1} \phi_{lk}(\rho_{ki}^{-1}\tau_{lkj}\phi_{kj}(\rho_{ji})\rho_{kj}) \rho_{lk} ) \\
&= \tau'_{lki} \phi'_{lk}(\tau'_{kji})
\end{split}
\]

(2) We note that the twisting functions $\tau_{kji}$ for the modules $\mathcal M_i$ 
do not necessarily belong to the center of the function rings $\mathcal A_k$.
In order to make a compatible theory, we require that the gluings $\phi_{ji}$ of function rings $\mathcal A_i$ are also twisted 
by the adjoint action of $\tau_{kji}$.

The equivalence of twistings are defined in coordination with the categorical equivalence of $\text{Qcoh}(\mathcal A, \tau)$.
But the category itself do not appear in the argument of the theorem.
\end{Rem}

\vskip 1pc

We extend our definition of NC deformations of an algebraic variety $X = \bigcup U_i$ with
$U_i = \text{Spec}(A_i)$ to twisted NC schemes as follows.

\begin{Defn}
A {\em twisted NC deformation} of an algebraic variety $X$ over an Artin local algebra $(R, M)$ is defined to be
a triple $(\mathcal A, \alpha, \tau)$ which satisfies the following conditions:
 
\begin{enumerate}

\item $\mathcal A$ is a collection of associative $R$-algebras $\mathcal A_i$ for $i \in I$ which are flat over $R$.

\item $\alpha$ is a collection of $k$-algebra isomorphisms $\alpha_i: k \otimes_R \mathcal A_i \to A_i$.

\item $\phi^{\mathcal A}$ is a collection of $R$-algebra homomorphisms
$\phi^{\mathcal A}_{ji}: \mathcal A_i \to \mathcal A_j$ for $i < j$.
We set $\phi^{\mathcal A}_{ii} = \text{Id}_{\mathcal A_i}$ for $i = j$.

\item $\tau$ is a collection of elements $\tau_{kji} \in \mathcal A_k$ for $i < j < k$ such that 
$\tau_{kji} \equiv 1 \in A_k$ under $\alpha_k$
and $\tau_{lji}\tau_{lkj} = \tau_{lki}\phi^{\mathcal A}_{lk}(\tau_{kji})$ for $i < j < k < l$.
We set $\tau_{kji} = 1$ if two among $i,j,k$ coincide.

\item $\phi^{\mathcal A}_{kj}\phi^{\mathcal A}_{ji}(x) = \tau_{kji}^{-1}\phi^{\mathcal A}_{ki}(x)\tau_{kji}$ for $i < j < k$ and 
$x \in \mathcal A_i$.

\item $\alpha_j \circ (k \otimes \phi^{\mathcal A}_{ji}) = \phi^A_{ji} \circ \alpha_i$.

\end{enumerate}
\end{Defn}

We note that the $\tau_{kji}$ are automatically invertible because the maximal ideal $M \subset R$ is nilpotent.

An {\em equivalence} of twisted NC deformations 
\[
(\mathcal A, \alpha, \tau) \sim (\mathcal A', \alpha', \tau')
\]
is an equivalence of twisted NC schemes $(\mathcal A, \tau) \sim (\mathcal A', \tau')$  
which is compatible with $\alpha, \alpha'$, i.e., 
$\alpha = \alpha' \circ (k \otimes \epsilon)$.
In this case, there is an equivalence $\Phi: \text{Qcoh}(\mathcal A, \tau) \to \text{Qcoh}(\mathcal A', \tau')$. 

\vskip 1pc

We have a twisted NC deformation functor $\text{Def }^t_X: (\text{Art}) \to (\text{Set})$ 
defined by $\text{Def }^t_X(R) = \{(\mathcal A, \alpha, \tau)\}/\sim$, 
where $\sim$ denotes an equivalence.

Theorem \ref{T1T2} is generalized for the twisted NC deformations:

\begin{Thm}\label{T1T2'}
Let $(R,M), (R',M') \in (\text{Art})$ and $J$ be as in (\ref{ext}), and
let $(\mathcal A, \alpha, \tau)$ be a twisted NC deformation of $X$ over $(R,M)$.
Then the following hold.

(1) There is a class $\xi^{(3,0)} \in J \otimes H^0(X, \bigwedge^3 T_X)$.
If $\xi^{(3,0)} = 0$, then there is a class $\xi^{(2,1)} \in J \otimes H^1(X, \bigwedge^2 T_X)$.
If $\xi^{(2,1)} = 0$, then there is a class $\xi^{(0,3)} \in J \otimes H^3(X, \mathcal O_X)$.
If $\xi^{(0,3)} = 0$, then there is a class $\xi^{(1,2)} \in J \otimes H^2(X, T_X)$.
And $\xi^{(3,0)} = \xi^{(2,1)} = \xi^{(1,2)} = \xi^{(0,3)} = 0$ hold if and only if there exists 
a twisted NC deformation $(\mathcal A', \alpha', \tau')$
over $(R',M')$ which induces $(\mathcal A, \alpha, \tau)$ over $(R,M)$.

(2) Assume that $\xi^{(3,0)} = \xi^{(2,1)} = \xi^{(1,2)} = \xi^{(0,3)} = 0$.
Then the set of equivalence classes of all extensions $(\mathcal A, \alpha, \tau)$ 
corresponds bijectively to a $k$-vector space 
$J \otimes (H^0(X, \bigwedge^2 T_X) \oplus H^1(X, T_X) \oplus H^2(X, \mathcal O_X))$.

(3) The obstruction classes $\xi^{(3,0)}, \xi^{(2,1)}, \xi^{(1,2)}, \xi^{(0,3)}$ are 
functorially determined in the following sense.
Let 
\begin{equation}\label{functorial}
\begin{CD}
0 @>>> J @>>> R' @>>> R @>>> 0 \\
@. @V{\beta_J}VV @V{\beta'}VV @V{\beta}VV \\
0 @>>> J_1 @>>> R'_1 @>>> R_1 @>>> 0
\end{CD}
\end{equation}
be a commutative diagram of extensions as in (\ref{ext}).
Let $(\mathcal A, \alpha, \tau) \in \text{Def }^t_X(R)$ 
and let $(\mathcal A^{(1)}, \alpha^{(1)}, \tau^{(1)}) 
= \text{Def }^t_X(\beta)(\mathcal A) \in \text{Def }^t_X(R_1)$.
Let $\xi^{(p,q)} \in J \otimes H^q(\bigwedge^p T_X)$ and 
$(\xi^{(1)})^{(p,q)} \in J_1 \otimes H^q(\bigwedge^p T_X)$ be the 
obstruction classes of extending $(\mathcal A, \alpha, \tau)$ and 
$(\mathcal A^{(1)}, \alpha^{(1)}, \tau^{(1)})$ 
over $R'$ and $R'_1$, respectively.
Then $(\xi^{(1)})^{(p,q)} = \beta_J(\xi^{(p,q)})$ hold as long as they are defined.
\end{Thm}

\begin{proof}
We will only explain the places where we need modifications from the proof of Theorem \ref{T1T2}.

\vskip 1pc

(1) We let $\mathcal A'_i = R' \otimes_k A_i$ as $R'$-modules with bilinear multiplications $\times_{R'_i}$ and gluings $\phi'_{ji}$
as in the proof of Theorem \ref{T1T2}.
Moreover let $\tau'_{kji} \in \mathcal A'_{kji}$ be elements which induce given $\tau_{kji} \in \mathcal A_{kji}$.

$f_i$ (for associativity of $\times_{R'_i}$), $g_{ji}$ (for compatibility of $\times_{R'_i}$ and $\phi'_{ji}$), $b_i$ (for different $\times_{R'_i}$),  
and $c_{ji}$ (for different $\phi'_{ji}$) are defined in the same formulas.
But we modify the definition of $h_{kji}: A_i \to J \otimes A_k$ (for cocycle condition of $\phi'_{ji}$) by
\[
h_{kji}(\bar x) = \phi'_{kj}\phi'_{ji}(x) - (\tau'_{kji})^{-1} \times_{R'_k} \phi'_{ki}(x) \times_{R'_k} \tau'_{kji}
\]
for $x \in \mathcal A'_i$ and $i < j < k$.
Moreover we consider $t_{kji} \in J \otimes A_k$ for different choice of twists:
\[
\tilde{\tau}'_{kji} = \tau'_{kji} + t_{kji}
\]
and $\sigma_{lkji} \in J \otimes A_l$ for the compatibility condition of the twists:
\[
\sigma_{lkji} = \tau'_{lji} \times_{R'_l} \tau'_{lkj} - \tau'_{lki}\times_{R'_l} \phi'_{lk}(\tau'_{kji}) = 
\tau'_{lji} \times_{R'_l} \tau'_{lkj} \times_{R'_l} (\phi'_{lk}(\tau'_{kji}))^{-1} \times_{R'_l} (\tau'_{lki})^{-1} - 1
\]
for $i < j < k < l$, where the second equality is a consequence of the fact that, if $z \in J \otimes \mathcal A_l$ and
$\tau - 1 \in M' \otimes \mathcal A_l$, then $z\tau = z$.

\vskip 1pc

Lemma \ref{df} (1), (2) and Lemma \ref{f'} (1), (2) hold without change, but 
we need to modify Lemma \ref{df} (3), (4) and Lemma \ref{f'} (3). 
Therefore the obstruction $\xi^{(3,0)}$ is defined without change.

We can assume from now that $\xi^{(3,0)} = 0$.
Thus we assume that $f_i = 0$, $db_i = 0$, $dg_{ji} = 0$, and $\times_{R'_i}$ is 
an associative multiplication on $\mathcal A_i$.

\vskip 1pc

The statement of Lemma \ref{df} (3) also holds without change, but the proof need to be modified as follows:
Since $\tau'_{kji} \equiv 1 \mod M'$ and $\times_{R'_k}$ is associative, we have
\[
\begin{split}
&\phi'_{kj}\phi'_{ji}(x \times_{R'_i} y) 
= \phi'_{kj}(\phi'_{ji}(x) \times_{R'_j} \phi'_{ji}(y) + g_{ji}(\bar x,\bar y)) \\
&= \phi'_{kj}\phi'_{ji}(x) \times_{R'_k} \phi'_{kj}\phi'_{ji}(y) + g_{kj}(\phi^0_{ji}(\bar x),\phi^0_{ji}(\bar y)) 
+ \phi^0_{kj}g_{ji}(\bar x,\bar y) \\
&= ((\tau'_{kji})^{-1} \times_{R'_k} \phi'_{ki}(x) \times_{R'_k} \tau'_{kji} + h_{kji}(\bar x)) \times_{R'_k} 
((\tau'_{kji})^{-1} \times_{R'_k} \phi'_{ki}(y) \times_{R'_k} \tau'_{kji} + h_{kji}(\bar y)) \\
&+ g_{kj}(\phi^0_{ji}(\bar x),\phi^0_{ji}(\bar y)) + \phi^0_{kj}g_{ji}(\bar x,\bar y) \\
&= (\tau'_{kji})^{-1} \times_{R'_k} \phi'_{ki}(x) \times_{R'_k} \phi'_{ki}(y) \times_{R'_k} \tau'_{kji} 
+ \phi^0_{ki}(\bar x)h_{kji}(\bar y) + h_{kji}(\bar x)\phi^0_{ki}(\bar y) \\
&+ g_{kj}(\phi^0_{ji}(\bar x),\phi^0_{ji}(\bar y)) + \phi^0_{kj}g_{ji}(\bar x,\bar y) \\
&= (\tau'_{kji})^{-1} \times_{R'_k} \phi'_{ki}(x \times_{R'_i} y) \times_{R'_k} \tau'_{kji} - g_{ki}(\bar x,\bar y) + \phi^0_{ki}(\bar x)h_{kji}(\bar y) 
+ h_{kji}(\bar x) \phi^0_{ki}(\bar y) \\
&+ g_{kj}(\phi^0_{ji}(\bar x),\phi^0_{ji}(\bar y)) + \phi^0_{kj}g_{ji}(\bar x,\bar y).
\end{split}
\]
Hence
\[
\begin{split}
&dh_{kji}(\bar x,\bar y) := \phi^0_{ki}(\bar x) h_{kji}(\bar y) - h_{kji}(\bar x \bar y) + h_{kji}(\bar x)\phi^0_{ki}(\bar y) \\
&= - g_{kj}(\phi^0_{ji}(\bar x),\phi^0_{ji}(\bar y)) + g_{ki}(\bar x,\bar y) - \phi^0_{kj}g_{ji}(\bar x,\bar y).
\end{split}
\]
that is
\[
dh_{kji} = - g_{kj}(\phi^0_{ji})^{\otimes 2} + g_{ki} - \phi^0_{kj}g_{ji}.
\]
This is the same statement as Lemma \ref{df} (3).
Then the \v Cech cocycle $v(g_{ji})$ is well-defined for arbitrary pairs $i,j$, and determines an obstruction class
$\xi^{(2,1)}$.

We assume from now that $\xi^{(2,1)} = 0$.
Then we assume that $g_{ji} = 0$, $dc_{ji} = 0$, $dh_{kji} = 0$, 
the gluings $\phi'_{ji}$ are compatible with the multiplications $\times_{R'_i}$.
We will omit to write $\times_{R'_i}$ from now on.

\vskip 1pc

We will consider the obstruction $\sigma_{lkji}$ for the extension of twists before considering the obstructions $f_{kji}$.

The following lemmas show that the \v Cech $3$-cochain $\{\sigma_{lkji}\}$ is closed and determined up to coboundaries
as long as $i < j < k < l < m$:

\begin{Lem}
Let $i < j < k < l < m$.
Then 
\[
\phi^0_{ml}(\sigma_{lkji}) - \sigma_{mkji} + \sigma_{mlji} - \sigma_{mlki} + \sigma_{mlkj} = 0.
\]
\end{Lem}

\begin{proof}
We use the following identities in the calculation:
if $\tau \equiv 1 \mod M'\mathcal A_l$ and $z,z' \in J \otimes \mathcal A_l$, then
$\tau^{-1}(1+z)\tau = 1+z$ and $(1+z)^{-1} = 1 - z$, $(1+z)(1+z') = 1+z+z'$, etc. 
For example, if $\tau_i \equiv 1 \mod M'\mathcal A_l$ and $\tau_1 \dots \tau_m - 1 \in J \otimes \mathcal A_l$, then
we have a cyclic permutation $\tau_1 \tau_2 \dots \tau_m - 1 = \tau_2 \dots \tau_m \tau_1 - 1$.
The left hand side of the claim is equal to:
\[
\begin{split}
&= \phi'_{ml}(\sigma_{lkji}) - \sigma_{mkji} + \sigma_{mlji} - \sigma_{mlki} + \sigma_{mlkj} \\
&= \phi'_{ml}(\tau'_{lji}) \phi'_{ml}(\tau'_{lkj}) - \phi'_{ml}(\tau'_{lki}) (\tau'_{mlk})^{-1} \phi'_{mk}(\tau'_{kji})\tau'_{mlk} \\
&- (\tau'_{mji} \tau'_{mkj} - \tau'_{mki} \phi'_{mk}(\tau'_{kji})) 
+ \tau'_{mji} \tau'_{mlj} - \tau'_{mli} \phi'_{ml}(\tau'_{lji}) \\
&- (\tau'_{mki} \tau'_{mlk} - \tau'_{mli} \phi'_{ml}(\tau'_{lki})) 
+ \tau'_{mkj} \tau'_{mlk} - \tau'_{mlj} \phi'_{ml}(\tau'_{lkj}), 
\end{split}
\]
because $\phi'_{ml}\phi'_{lk}(\tau'_{kji}) = (\tau'_{mlk})^{-1} \phi'_{mk}(\tau'_{kji})\tau'_{mlk}$, 
where we note that $h_{mlk}(1) = 0$.

Let $\tau, \tau' \equiv 1 \mod M'\mathcal A_l$ and $z,z' \in J \otimes \mathcal A_l$. 
If $\tau - \tau' = z$, then $z = \tau (\tau')^{-1} - 1$.
Indeed we have $z = z \tau' = \tau - \tau'$.
If $z = \tau - 1$ and $z' = \tau' - 1$, then $z + z' = \tau \tau' - 1$.
Indeed we have $1 + z + z' = (1+z)(1+z') = \tau \tau'$.
Therefore the last formula is equal to:
 \[
\begin{split}
&=\phi'_{ml}(\tau'_{lji}) \phi'_{ml}(\tau'_{lkj}) (\tau'_{mlk})^{-1} (\phi'_{mk}(\tau'_{kji}))^{-1} \tau'_{mlk} (\phi'_{ml}(\tau'_{lki}))^{-1} \\
&\times \tau'_{mki} \phi'_{mk}(\tau'_{kji}) (\tau'_{mkj})^{-1} (\tau'_{mji})^{-1}
\times \tau'_{mji} \tau'_{mlj} (\phi'_{ml}(\tau'_{lji}))^{-1} (\tau'_{mli})^{-1} \\
&\times \tau'_{mli}\phi'_{ml}(\tau'_{lki}) (\tau'_{mlk})^{-1}(\tau'_{mki})^{-1} 
\times \tau'_{mkj} \tau'_{mlk} (\phi'_{ml}(\tau'_{lkj}))^{-1} (\tau'_{mlj})^{-1} - 1,
\end{split}
\]
where we divided the product by $\times$ to the factors which belong to $1 + J \otimes \mathcal A_l$. 
This is equal by the cyclic rotation to:
\[
\begin{split}
&=\phi'_{ml}(\tau'_{lji}) \phi'_{ml}(\tau'_{lkj}) (\tau'_{mlk})^{-1} (\phi'_{mk}(\tau'_{kji}))^{-1} 
\tau'_{mlk} (\phi'_{ml}(\tau'_{lki}))^{-1} \\
&\times \phi'_{ml}(\tau'_{lki}) (\tau'_{mlk})^{-1}(\tau'_{mki})^{-1} \tau'_{mli}
\times \tau'_{mkj} \tau'_{mlk} (\phi'_{ml}(\tau'_{lkj}))^{-1} (\tau'_{mlj})^{-1} \\
&\times \tau'_{mki} \phi'_{mk}(\tau'_{kji}) (\tau'_{mkj})^{-1} (\tau'_{mji})^{-1}
\times \tau'_{mji} \tau'_{mlj} (\phi'_{ml}(\tau'_{lji}))^{-1} (\tau'_{mli})^{-1} - 1.
\end{split}
\]
Using $\tau'_{mlk} (\phi'_{ml}(\tau'_{lki}))^{-1}\phi'_{ml}(\tau'_{lki}) (\tau'_{mlk})^{-1} = 1$ and the cyclic rotation, we have
\[
\begin{split}
&= \phi'_{ml}(\tau'_{lkj}) (\tau'_{mlk})^{-1} (\phi'_{mk}(\tau'_{kji}))^{-1} (\tau'_{mki})^{-1} \tau'_{mli} \phi'_{ml}(\tau'_{lji}) \\
&\times (\phi'_{ml}(\tau'_{lji}))^{-1} (\tau'_{mli})^{-1} \tau'_{mji} \tau'_{mlj} 
\times \tau'_{mkj} \tau'_{mlk} (\phi'_{ml}(\tau'_{lkj}))^{-1} (\tau'_{mlj})^{-1}\\
&\times \tau'_{mki} \phi'_{mk}(\tau'_{kji}) (\tau'_{mkj})^{-1} (\tau'_{mji})^{-1} - 1.
\end{split}
\]
Using $\tau'_{mli} \phi'_{ml}(\tau'_{lji})(\phi'_{ml}(\tau'_{lji}))^{-1} (\tau'_{mli})^{-1} = 1$ and the cyclic rotation
of $\phi'_{ml}(\tau'_{lkj}) (\tau'_{mlk})^{-1}$ and $\tau'_{mkj}$, we have
\[
\begin{split}
&= (\phi'_{mk}(\tau'_{kji}))^{-1} (\tau'_{mki})^{-1} \tau'_{mji} \tau'_{mlj}\phi'_{ml}(\tau'_{lkj}) (\tau'_{mlk})^{-1} \\
&\times \tau'_{mlk} (\phi'_{ml}(\tau'_{lkj}))^{-1} (\tau'_{mlj})^{-1}\tau'_{mkj} 
\times \tau'_{mki} \phi'_{mk}(\tau'_{kji}) (\tau'_{mkj})^{-1} (\tau'_{mji})^{-1} - 1. 
\end{split}
\]
And similarly, 
\[
\begin{split}
&= (\phi'_{mk}(\tau'_{kji}))^{-1} (\tau'_{mki})^{-1} \tau'_{mji} \tau'_{mkj} 
\times (\tau'_{mkj})^{-1} (\tau'_{mji})^{-1} \tau'_{mki} \phi'_{mk}(\tau'_{kji}) - 1 \\
&= 1 - 1 = 0.
\end{split}
\]
\end{proof}

We have the following lemma for the change of choices of extended twists:

\begin{Lem}
Let $\tilde{\sigma}_{lkji} \in J \otimes A_l$ be the compatibility cochain for a different choice of twists 
$\tilde{\tau}'_{kji} = \tau'_{kji} + t_{kji}$.
Then the obstruction cochain is changed as follows:
\[
\tilde{\sigma}_{lkji} - \sigma_{lkji} = - \phi^0_{lk}(t_{kji}) + t_{lji} - t_{lki} + t_{lkj}
\]
for $i < j < k < l$.
\end{Lem}

\begin{proof}
\[
\begin{split}
&(\tau'_{lji}+t_{lji})(\tau'_{lkj}+t_{lkj}) - (\tau'_{lki}+t_{lki})\phi'_{lk}(\tau'_{kji}+t_{kji})
- (\tau'_{lji}\tau'_{lkj} - \tau'_{lki}\phi'_{lk}(\tau'_{kji})) \\
&= t_{lji} + t_{lkj} - t_{lki} - \phi^0_{lk}(t_{kji}).
\end{split}
\]
\end{proof}

\begin{Rem}
We do not know whether the statements of the above lemmas follow from a long exact sequence associated to a 
short exact sequence
\[
0 \to J \otimes \mathcal O_X \to (\mathcal A')^{\times} \to \mathcal A^{\times} \to 1
\]
because $\mathcal A^{\times}$, etc. are non-abelian. 
\end{Rem}

We can extend the definition of the $\sigma_{i_0i_1i_2i_3}$ for arbitrary indexes $i_0,i_1,i_2,i_3$ by Lemma \ref{S_n}.
Thus we define the obstruction class $\xi^{(0,3)}$ whose vanishing is equivalent to the extendability
of the twists.

\vskip 1pc

Now let us assume that the cohomology class $\xi^{(0,3)} = 0$.
Then we may assume that $\sigma_{lkji} = 0$ for all $i,j,k,l$.
We will lastly treat the obstruction $h_{kji}$.

The following is the modifications of Lemma \ref{df} (4) and Lemma \ref{f'} (3):

\begin{Lem}\label{df+f'} 
Assume $i < j < k < l$. 
Then:

(1) $\phi^0_{lk}h_{kji} - h_{lji} + h_{lki} - h_{lkj}\phi^0_{ji} = 0$. 

(2) $h'_{kji} = h_{kji} + \phi^0_{kj}c_{ji} - c_{ki} + c_{kj}\phi^0_{ji}$.
\end{Lem}

\begin{proof}
(1) If $z \in 1 + J \otimes A_l$ and $\tau \equiv 1 \mod M'\mathcal A_l$, then $z = \tau^{-1}z\tau$.
Hence
\[
\begin{split}
&\phi^0_{lk}h_{kji}(\bar x) - h_{lji}(\bar x) + h_{lki}(\bar x) - h_{lkj}(\phi^0_{ji}(\bar x)) \\
&= \phi'_{lk}(\phi'_{kj}\phi'_{ji}(x) - (\tau_{kji})^{-1}\phi'_{ki}(x)\tau_{kji}) - (\phi'_{lj}\phi'_{ji}(x) - (\tau_{lji})^{-1}\phi'_{li}(x)\tau_{lji}) \\
&+ (\phi'_{lk}\phi'_{ki}(x) - (\tau_{lki})^{-1}\phi'_{li}(x)\tau_{lki}) 
- (\phi'_{lk}\phi'_{kj}\phi'_{ji}(x) - (\tau_{lkj})^{-1}\phi'_{lj}\phi'_{ji}(x)\tau_{lkj}) \\
&= \phi'_{lk}\phi'_{kj}\phi'_{ji}(x) - \phi'_{lk}(\tau_{kji})^{-1}\phi'_{lk}\phi'_{ki}(x)\phi'_{lk}(\tau_{kji}) \\
&- (\tau_{lkj})^{-1}(\phi'_{lj}\phi'_{ji}(x) - (\tau_{lji})^{-1}\phi'_{li}(x)\tau_{lji})\tau_{lkj} \\
&+ \phi'_{lk}(\tau_{kji})^{-1}(\phi'_{lk}\phi'_{ki}(x) - (\tau_{lki})^{-1}\phi'_{li}(x)\tau_{lki}) \phi'_{lk}(\tau_{kji}) \\
&- (\phi'_{lk}\phi'_{kj}\phi'_{ji}(x) - (\tau_{lkj})^{-1}\phi'_{lj}\phi'_{ji}(x)\tau_{lkj}) \\
&= (\tau_{lkj})^{-1} (\tau_{lji})^{-1}\phi'_{li}(x)\tau_{lji}\tau_{lkj} 
- \phi'_{lk}(\tau_{kji})^{-1}(\tau_{lki})^{-1}\phi'_{li}(x)\tau_{lki} \phi'_{lk}(\tau_{kji}) \\
&= 0
\end{split}
\]
because $\tau_{lji}\tau_{lkj} = \tau_{lki} \phi'_{lk}(\tau_{kji})$ by $\sigma_{lkji} = 0$.

(2) We calculate
\[
\begin{split}
&h'_{kji}(\bar x) = \phi'_{kj}(\phi'_{ji}(x) + c_{ji}(\bar x)) + c_{kj}(\phi^0_{ji}(\bar x)) - \tau_{kji}^{-1}(\phi'_{ki}(x) + c_{ki}(\bar x))\tau_{kji} \\
&= h_{kji}(\bar x) + \phi^0_{kj}(c_{ji}(\bar x)) - c_{ki}(\bar x) + c_{kj}(\phi^0_{ji}(\bar x)).
\end{split}
\]
\end{proof}

The rest of the proof of the part (1) of Theorem \ref{T1T2'} is the same as in Theorem \ref{T1T2}.

\vskip 1pc

(2) The additional choice of extensions comes from the change of twisting given by 
$\tau'_{kji} \mapsto \tau'_{kji} + t_{kji}$ for $t_{kji} \in J \otimes \mathcal A_k$.
The \v Cech $2$-cochain $\{t_{kji}\}$ is closed if and only if the condition $\sigma_{lkji} = 0$ is preserved.

We note that this change of twisting does not affect the gluing of algebras $\mathcal A_i$.
Indeed $(\tau'_{kji} + t)^{-1}x(\tau'_{kji} + t) = (\tau'_{kji})^{-1}x\tau'_{kji} + \bar xt - t\bar x = (\tau'_{kji})^{-1}x\tau'_{kji}$
for $x \in \mathcal A_k$, because $A_k$ is commutative.

We claim that $\{t_{kji}\}$ is a \v Cech coboundary if and only if there is an equivalence of twisted schemes
$(\mathcal A', \tau') \cong (\mathcal A', \tau' + t)$. 
More precisely there is a \v Cech cocycle $\{s_{ji}\}$ with $s_{ji} \in J \otimes \mathcal A_j$ such that 
$t_{kji} = \phi^0_{kj}(s_{ji}) - s_{ki} + s_{kj}$ if and only if there is an equivalence given by $\rho_{ji} = 1 + s_{ji}$.

Indeed we have
$(1+t_{kji}) = (1 - s_{ki})\phi'_{kj}(1+s_{ji})(1+s_{kj})$, i.e.,
$t_{kji} = - s_{ki} + \phi^0_{kj}(s_{ji}) + s_{kj}$.

\vskip 1pc

The proof of (3) is the same as in Theorem \ref{T1T2}.
\end{proof}

The existence of a semi-universal twisted NC deformation is proved in a similar way, and we obtain the 
following corollary:

\begin{Thm}\label{semi-universal'}
Let $X$ be a smooth algebraic variety, and
assume that 
\[
\begin{split}
&n = \dim H^0(X, \bigwedge^2 T_X) + \dim H^1(X, T_X) + \dim H^2(X, \mathcal O_X), \\
&m = \dim H^0(X, \bigwedge^3 T_X) + \dim H^1(X, \bigwedge^2 T_X) + \dim H^2(X, T_X) + \dim H^3(X, \mathcal O_X)
\end{split}
\]
are finite.
Then there exists a semi-universal deformation for the twisted NC deformation functor $\text{Def}\,^t_X$ over 
\[
R := k[[t_1,\dots,t_n]]/(f_1,\dots,f_m)
\]
for some formal power series $f_i$ of orders in $[2, \infty]$.
\end{Thm}


Graduate School of Mathematical Sciences, University of Tokyo,
Komaba, Meguro, Tokyo, 153-8914, Japan. 

kawamata@ms.u-tokyo.ac.jp


\begin{thebibliography}{}

\bibitem{DVLL}
Dinh Van, Hoang; Liu, Liyu; Lowen, Wendy.
{\em Non-commutative deformations and quasi-coherent modules}.
Selecta Math. {\bf 23} (2016), 1061--1119.

\bibitem{NC base}
Kawamata, Yujiro.
{\em Deformations over non-commutative base}.
to appear in C. R. Math.

\bibitem{Lowen-VdBergh}
Lowen, Wendy; Van den Bergh, Michel. 
{\em Deformation theory of abelian categories}.
Trans. Amer. Math. Soc. {\bf 358} (2006), 5441--5483.

\bibitem{Schlessinger}
Schlessinger, Michael.
{\em Functors of Artin rings}.
Trans. Amer. Math. Soc. {\bf 130}, No. 2 (Feb., 1968), pp. 208--222.

\bibitem{Sernesi}
Sernesi, Edoardo.
{\em Deformations of Algebraic Schemes}.
Grundlehren math. Wissenshaften {\bf 334} (2006), Springer.

\bibitem{Toda}
Toda, Yukinobu.
{\em Deformations and Fourier-Mukai transforms}.
J. Differential Geom. {\bf 81}(1): 197--224 (January 2009). DOI: 10.4310/jdg/1228400631. 

\bibitem{W}
Weibel, Charles.
{\em An Introduction to Homological Algebra}.
Cambridge St. Adv. Math. {\bf 38} (1994), Cambridge U. P.

\end{thebibliography}
\end{document}